\documentclass[12pt]{amsart}

\usepackage{amssymb,latexsym,amsmath,extarrows, mathrsfs, amsthm}
\usepackage{graphicx}

\usepackage[margin=1in, centering]{geometry}

\usepackage{tikz}
\usepackage{pgfplots}
\pgfplotsset{compat=newest}

\usepackage{pgflibraryarrows}
\usepackage{pgflibrarysnakes}

\newtheorem{theorem}{Theorem}[section]
\newtheorem{lemma}[theorem]{Lemma}
\newtheorem{proposition}[theorem]{Proposition}

\newtheorem{definition}[theorem]{Definition}

\newtheorem{claim}[theorem]{Claim}

\newtheorem*{acknowledgement}{Acknowledgement}

\newcommand{\al}{\alpha}
\newcommand{\be}{\beta}

\newcommand{\La}{\Lambda}

\newcommand{\vp}{\varphi}
\newcommand{\om}{\omega}
\newcommand{\Om}{\Omega}

\newcommand{\cs}{\mathcal S}

\newcommand{\cj}{\mathcal J}

\newcommand{\ff}{\mathscr{F}}

\newcommand{\nf}{\infty}

\newcommand{\ZR}{\mathbb{R}}
\newcommand{\ZZ}{\mathbb{Z}}
\newcommand{\ZN}{\mathbb{N}}

\newcommand{\bT}{{\bf T}}

\newcommand{\tM}{{\mathcal M}}
\newcommand{\hichi}{\raisebox{0.7ex}{\(\chi\)}}

\begin{document}

\title[the bilinear Hilbert transform]{Quasi pieces of the bilinear Hilbert transform incorporated into a paraproduct}
\author{Dong Dong}
\address{Mathematics Department\\
University of Illinois at Urbana-Champaign\\
Urbana, IL 61801}
\email{ddong3@illinois.edu}
\date{\today}

\begin{abstract}
We prove the boundedness of a class of tri-linear operators consisting of a quasi piece of bilinear Hilbert transform whose scale equals to or dominates the scale of its linear counter part. Such type of operators is motivated by the tri-linear Hilbert transform and its curved versions. 
\end{abstract}

\let\thefootnote\relax
\let\thefootnote\relax

\maketitle

\section{Introduction}
\setcounter{equation}0
\subsection{Background}
In a pair of breakthrough papers \cite{LT97,LT99}, Lacey and Thiele proved the boundedness property of the bilinear Hilbert transform (BHT) 

$$
B(f_1,f_2)(x)=p.v.\int f_1(x-t)f_2(x+t)\frac{1}{t}\,dt.
$$
Many interesting results about multilinear operators have been established in the spirit of Lacey-Thiele's method. 
However, $L^p$-boundedness of tri-linear Hilbert transform (THT)

$$
T(f_1,f_2,f_3)(x)=p.v.\int f_1(x-t)f_2(x-2t)f_3(x-3t)\frac{1}{t}\,dt.
$$
is still unknown. One difficulty arises from certain non-linear issue hidden in the trilinear structure.
This is one of the main reasons motivating Li to study BHT along curves \cite{LAP}, say 

$$
H_\Gamma (f_1,f_2)(x)=p.v.\int f_1(x-t)f_2(x-t^d)\frac{1}{t}\, dt, \text{ where } d\ge 2 \text{ is an integer}.
$$

In \cite{LAP}, $H_\Gamma$ is split into  two operators according to the efficiency of some oscillatory integral estimate (stationary phase vs. non-stationary phase). One of the two operators is a paraproduct of the form $\Pi_\Gamma(f_1,f_2)=\sum_k f_{1k}f_{2k}$ \cite{LNYJ} that is more complex than the classical Coifman-Meyer paraproduct \cite{CM91}. Although it turns out $\Pi_\Gamma$ is slightly simpler than BHT, the proof of its boundedness already requires sophisticated multi-scale time-frequency analysis that is essential in the study of BHT. Hence it is reasonable to expect that tri-linear analogues of 
the paraproduct $\Pi_\Gamma$ would be easier to handle than THT, but at the same time the study of such tri-linear operators could provide some new 
insights to THT.

The definition of tri-linear correspondence of $\Pi_\Gamma(f_1,f_2)$ was given in \cite{DL}, where the author and Li introduced the following class of operators $T^{\alpha,\beta}$ that can be viewed a hybrid of BHT and paraproduct:

\begin{equation} \label{eq: T}
T^{\alpha,\beta}(f_1,f_2,f_3)(x)=\sum_{k\in\ZZ}H^{\alpha, k}(f_1,f_2)(x)f_{3}^{\beta,k}(x), 
\end{equation} where

\begin{equation} 
\begin{cases}
H^{\alpha,k}(f_1,f_2)(x)=\iint_{\ZR^2} \widehat{f_1}(\xi_1)\widehat{f_2}(\xi_2)e^{2\pi i (\xi_1+\xi_2)x}\widehat{\Phi_1}\left(\frac{\xi_1-\xi_2}{2^{\al k}}\right) \, d\xi_1d\xi_2,  \label{eq: H}\\
\\
f^{\beta,k}(x)=\int_{\ZR}\widehat{f}(\xi)e^{2\pi i\xi x}\widehat{\Phi_2}\left(\frac{\xi}{2^{\beta k}}\right)\, d\xi. 
\end{cases}
\end{equation}

\noindent Here $\al, \beta$ are non-zero positive real numbers, and various conditions (about smoothness, support, etc) can be imposed on the cut-off functions $\widehat{\Phi_1}$ and $\widehat{\Phi_2}$. 

$T^{\alpha, \beta}$ is closely related with THT along curves. For example, one promising way to prove the boundedness of $T_C(f_1,f_2,f_3)(x)=p.v.\int f_1(x-t)f_2(x+t)f_3(x-t^d)\frac{dt}{t}$ is to study $T^{1,d}$ first (See \cite{LAP} for a similar approach in the bilinear setting). The following theorem is proved in \cite{DL}.
\begin{theorem} [\cite{DL}, Theorem 1.2]\label{Thm:DL}
  Let $\Phi_1$ and $\Phi_2$ be smooth functions satisfying supp $\widehat{\Phi_1}\subseteq [9, 10]$ and supp $\widehat{\Phi_2}\subseteq [-1, 1]$. Assume $\al=\beta\neq 0$. Then the operator $T^{\alpha,\beta}$ defined by \eqref{eq: T}\eqref{eq: H} is bounded from $L^{p_1}\times L^{p_2}\times L^{p_3}$ to $L^p$, $\frac{1}{p}=\frac{1}{p_1}+\frac{1}{p_2}+\frac{1}{p_3}$, whenever $(p_1,p_2,p_3)\in D=\{(p_1,p_2,p_3)\in (1, \nf)^3: \frac{1}{p_1}+\frac{1}{p_2}<\frac{3}{2}\}$.
\end{theorem}

\begin{figure} [h]
\begin{center}
\begin{tikzpicture}

\draw[->] (-1.0,0) -- (4,0) coordinate (x axis);
\draw[->] (0,-1.0) -- (0,6) coordinate (y axis);

\node (rect) [rectangle, draw, minimum width=7mm, minimum height=8mm] at (2,3) {1};

\node (rect) [rectangle, draw, minimum width=7mm, minimum height=8mm] at (2,1.5) {2};

\node (rect) [rectangle, draw, minimum width=7mm, minimum height=16mm] at (2,0) {3};

\node (rect) [rectangle, draw, minimum width=7mm, minimum height=16mm] at (2,5) {4};

\end{tikzpicture}
\caption{Tile structure of $T^{\alpha,\beta}$, $\alpha <\beta$, $k\ge 2$}
 \label{fig: tile for T2}
\end{center}
\end{figure}
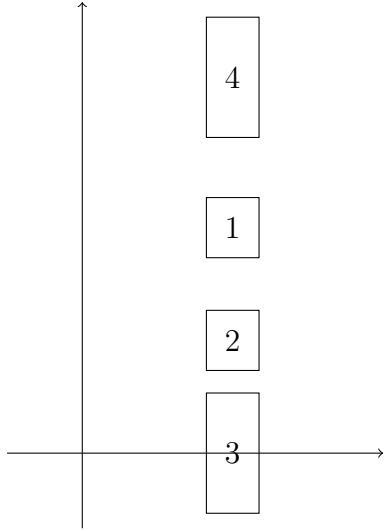
\noindent\textbf{Remarks.} (1) Strictly speaking, this theorem is proved in \cite{DL} only in the case $\al=\beta=1$, but this restriction is inessential. The proof given in \cite{DL} works for any homogeneous-scale case. 

(2) The intervals $[9,10]$ and $[-1,1]$ in the assumptions of Theorem \ref{Thm:DL} are not essential. The point is that $\widehat{\Phi_1}$ should be supported away from $0$ and $\widehat{\Phi_2}$ should be supported near $0$.

(3) We conjectured that the condition $\alpha=\beta$ can be dropped in the above theorem, but the proof given in \cite{DL} relies on the homogeneity of the scales. Let us briefly analyze the difficulties in the case $\alpha\neq \beta$ here. Assume $0<\alpha<\beta$ and let $k\ge 2$ be an integer. After wave packet decomposition, the tile associated with $f_{3}^{\beta,k}$ dominates the other two tiles (associated with $f_1$ and $f_2$) in frequency space as supp $\widehat{f_{3}^{\beta,k}}$ has a much larger scale $2^{\beta k}$. This will also introduce a long tile for the fourth function $f_4$ in the 4-linear form $\langle T^{\alpha}(f_1,f_2,f_3), f_4\rangle$: see Figure \ref{fig: tile for T2}. As there are two long tiles and one of them contains the origin, the situation is difficult to handle even we use telescoping techniques that are powerful in some uniform estimates (\cite{GL,LRev,Thiele 02}).

\subsection{Main result and application}

The purpose of this paper is to investigate other instances of $T^{\alpha,\beta}$, including some non-homogeneous-scale cases. We would like to switch the roles of $\widehat{\Phi_1}$ and $\widehat{\Phi_2}$, i.e. assume that $\widehat{\Phi_1}$ is supported near the origin and $\widehat{\Phi_2}$ is supported away from $0$ (instead of the other way around in Theorem \ref{Thm:DL}). In this case, $H^{\alpha,k}$ is no longer a piece of BHT at certain scale: we may call it a \textit{quasi piece} of BHT. Surprisingly we can obtain the same range of boundedness as before, even in some cases with non-homogeneous scales (See Theorem \ref{thm:non-homo} below). More precisely, we have

\begin{theorem} \label{thm: homo}
Let $\Phi_1$ and $\Phi_2$ be smooth bump functions satisfying supp $\widehat{\Phi_1}\subseteq [-1, 1]$  and supp $\widehat{\Phi_2}\subseteq [9, 10]$. Let $\alpha=\beta\neq 0$. Then the operator $T^{\alpha,\beta}$ defined by \eqref{eq: T}\eqref{eq: H} is bounded from $L^{p_1}\times L^{p_2}\times L^{p_3}$ to $L^p$ for any $(p_1,p_2,p_3)\in D=\{(p_1,p_2,p_3)\in (1, \nf)^3: \frac{1}{p_1}+\frac{1}{p_2}<\frac{3}{2}\}$, $\frac{1}{p_1}+\frac{1}{p_2}+\frac{1}{p_3}=\frac{1}{p}$.
\end{theorem}

The proof of Theorem \ref{thm: homo} uses Lacey-Thiele's ideas about BHT. However, it should be noted that because of the quasi pieces of BHT, the 4-tile structure of the operator $T^{\al,\al}$ quite different from the tri-tile structure of BHT (see Figure \ref{fig: compare} for a comparison): the loss of one tile ($1$-tile and $2$-tile are identical) forces us to mainly work with only two tiles as opposed to three tiles in BHT. The presence of a Littlewood-Paley piece ($3$-tile), however, will be of great help (see the proof of Proposition \ref{lemma: single tree}).

Using Theorem \ref{thm: homo} together with Theorem \ref{Thm:DL}, we can derive the boundedness property of positive truncations of $T^{\al, \beta}$ in some non-homogeneous-scale cases. 

\begin{theorem} \label{thm:non-homo}
Let $\Phi_1$ and $\Phi_2$ be smooth bump functions satisfying supp $\widehat{\Phi_1}\subseteq [-1, 1]$  and supp $\widehat{\Phi_2}\subseteq [9, 10]$. Assume $\al>\beta>0$. Define a positive truncation of $T^{\al, \beta}$ by

\begin{equation} \label{def positive part}
T^{\al,\beta}_N(f_1,f_2,f_3)(x)=\sum_{k\ge N}H^{\al,k}(f_1,f_2)(x)f_{3}^{\beta, k}(x),~N\in\ZN,
\end{equation}
where $H^{\al,k}$ and $f_3^{\beta,k}$ are given in \eqref{eq: H}. Then for any $N\ge 10\al/\beta$, the operator $T^{\al,\beta}_N$ is bounded from $L^{p_1}\times L^{p_2}\times L^{p_3}$ into $L^p$ for any $(p_1,p_2,p_3)\in D=\{(p_1,p_2,p_3)\in (1, \nf)^3: \frac{1}{p_1}+\frac{1}{p_2}<\frac{3}{2}\}$, $\frac{1}{p_1}+\frac{1}{p_2}+\frac{1}{p_3}=\frac{1}{p}$.
\end{theorem}

\noindent\textbf{Remarks.} (1) The choice of intervals $[-1,1]$ and $[9,10]$ in the above two theorems are not important. The key is that $\widehat{\Phi_1}$ should be supported near $0$ and $\widehat{\Phi_2}$ should be supported away from $0$.

(2) One of anticipated applications of Theorem \ref{thm:non-homo} is to use boundedness of $T^{d,1}_N$ to prove that of one prototype of THT along polynomial curves 

$$
T^C(f_1,f_2,f_3)(x)=p.v.\int_{-1}^1 f_1(x-t)f_2(x-t^d)f_3(x+t^d)\,\frac{dt}{t}.
$$ 
Just like the relationship between $H_\Gamma (f_1,f_2)(x)=p.v.\int f_1(x-t)f_2(x-t^d)\frac{1}{t}\, dt$ and the paraproduct $\Pi_\Gamma(f_1,f_2)=\sum_k f_{1k}f_{2k}$ studied in \cite{LNYJ}, $T^C$ can be written as the sum of finitely many operators of the form $T^{\al,\beta}_N$ (plus some other terms). The condition $N\ge 10\al/\beta$ in Theorem \ref{thm:non-homo} is assumed only for technical reasons and it does not affect the application as each scale of $T^C$ (after the standard dyadic decomposition $\frac{1}{t}=\sum_{k}\rho_k(t)$) is trivially bounded. The reason that we only consider the positive truncation instead of $T^{\al,\be}$ itself is that $|t|\le 1$ in the definition of $T^C$.

(3) Under the assumptions on $\widehat{\Phi_1}$ and $\widehat{\Phi_2}$ in Theorem \ref{thm:non-homo}, Figure \ref{fig: before tele} illustrates the worst case of the tri-tile structure of $T^{\al,\beta}_N$ with $\al>\be$ at any positive scale $k$. The two identical long tiles seems to be very problematic. The key to resolve this issue is to reduce the study of $T^{\al,\beta}_N$ with $\al>\beta$ to that of $T^{\beta,\beta}$ (homogeneous case) by a telescoping argument. The details are provided in Section \ref{section: telescoping}.

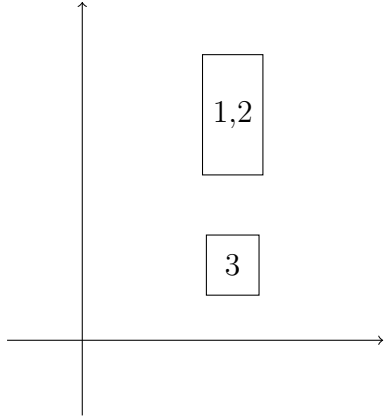
\begin{figure} [h]
\begin{center}
\begin{tikzpicture}

\draw[->] (-1.0,0) -- (4,0) coordinate (x axis);
\draw[->] (0,-1.0) -- (0,4.5) coordinate (y axis);

\node (rect) [rectangle, draw, minimum width=7mm, minimum height=8mm] at (2,1) {3};

\node (rect) [rectangle, draw, minimum width=7mm, minimum height=16mm] at (2,3) {1,2};

\end{tikzpicture}
\caption{tri-tile structure of $T^{\al,\beta}_N$, $\al>\beta>0$, $k\ge 1$}
\label{fig: before tele}
\end{center}
\end{figure}

\subsection{Notations}
%The rest of the paper is organized as follows. In section \ref{section: telescoping} we show how to derive Theorem \ref{thm:non-homo} using Theorem \ref{thm: homo} by a telescoping argument. The remaining four sections contain the proof of Theorem \ref{thm: homo}. 
Throughout the paper we will use $C$ to denote a positive constant whose value may change from line to line. We may add one or more subscripts to $C$ to emphasize dependence of $C$. $A\lesssim B$ is short for $A\le CB$ and $A\lesssim_N B$ means $A\le C_NB$. If $A\lesssim B$ and $B\lesssim A$, then we write $A\simeq B$. $\hichi_E$ and $|E|$ will be used to denote the characteristic function and the Lebesgue measure of the set $E$, respectively.

\section{Reduction to Model Form} \label{section: reduction}
\setcounter{equation}0
The goal of this section is to reduce Theorem \ref{thm: homo} to the study of a model form using standard wave packet decomposition process. For notational convenience, we assume $\al=\beta=1$ in the proof. The general case can be handled the same way. 

Let $\cs(\ZR)$ denote the class of Schwartz functions on $\ZR$. Given $f_j\in\cs(\ZR)$, $j\in\{1,2,3,4\}$, consider the 4-linear form $\La$ associated with $T^{1,1}$

\begin{equation} \label{intro of La}
\begin{split}
&\La (f_1,f_2,f_3,f_4):=\int T^{1,1}(f_1,f_2,f_3)(x)\overline{f_4}(x)\, dx\\
&=\sum_{k\in\ZZ}\iiint \widehat{f_1}(\xi_1)\widehat{f_2}(\xi_2)\widehat{f_3}(\xi_3)\widehat{\Phi_1}\left(\frac{\xi_1-\xi_2}{2^k}\right)\widehat{\Phi_2}\left(\frac{\xi_3}{2^k}\right)\overline{\widehat{f_4}}(\xi_1+\xi_2+\xi_3)\,d\xi_1d\xi_2d\xi_3,
\end{split}
\end{equation}
where supp $\widehat{\Phi_1}\subseteq [-1, 1]$  and supp $\widehat{\Phi_2}\subseteq [9, 10]$. 

To simplify the 4-linear form above, we use the wave packet decomposition. Choose a $\psi \in\cs(\ZR)$ such that supp$\widehat{\psi}\subseteq [0,1]$ and 

$$
\sum_{l\in\ZZ}\widehat{\psi}\left(\xi-\frac{l}{2}\right)=1 \text{ for any } \xi\in\ZR.
$$  
Define 

$$
\widehat{\psi_{k,l}}(\xi):=\widehat{\psi}\left(\frac{\xi-2^{k-1}l}{2^k}\right) \text{ for } (k,l)\in \ZZ^2. 
$$ 
Pick a non-negative $\vp\in\cs(\ZR)$ with supp$\widehat{\vp}\subseteq [-1,1]$ and $\widehat{\vp}(0)=1$. Let 

$$
\vp_k(x):=2^k\vp(2^kx), k\in\ZZ.
$$ 
For every $(k,n)\in\ZZ^2$, denote $I_{k,n}:=[2^{-k}n, 2^{-k}(n+1))$. Then for each scale $k\in\ZZ$ and any function $f\in\cs(\ZR)$, we have

\begin{equation} \label{eq: WPD}
f=\sum_{(n,l)\in\ZZ^2}f_{k,n,l},
\end{equation}
where 

\begin{align}
&f_{k,n,l}(x):=\hichi_{I_{k,n}}^*(x)f*\psi_{k,l}(x), \text{ and } \label{f}\\ 
&\hichi_I^*(x):=\hichi_I*\vp_k(x) \text{ for any interval } I. \label{chi I}
\end{align}

In sum, $f_{k,n,l}$ is well-localized, as supp $\widehat{f_{k,n,l}}\subseteq [2^k(\frac{l}{2}-1), 2^k(\frac{l}{2}+2)]$ and $f_{k,n,l}$ is essentially supported on $I_{k,n}$ in the sense that

\begin{equation}
|f_{k,n,l}(x)|\lesssim_{N,M} \left(1+\frac{\text{dist}(x, I_{k,n})}{|I_{k,n}|}\right)^{-N}\frac{1}{|I_{k,n}|}\int |f(y)| \left(1+\frac{|x-y|}{|I_{k,n}|} \right)^{-M} \,dy.
\end{equation} 

Now we apply the decomposition \eqref{eq: WPD} to all the four functions in \eqref{intro of La} and obtain 

\[
\begin{split}
\La(f_1,f_2,f_3,f_4)=&\sum_{\substack{k\in\ZZ \\(n_1,n_2, n_3,n_4)\in\ZZ^4\\(l_1,l_2, l_3,l_4)\in\ZZ^4}}\iiint  \widehat{(f_1)_{k,n_1,l_1},}(\xi_1)\widehat{(f_2)_{k,n_2,l_2}}(\xi_2)\widehat{(f_3)_{k,n_3,l_3}}(\xi_3) \\
& \widehat{\Phi_1}\left(\frac{\xi_1-\xi_2}{2^k}\right)\widehat{\Phi_2}\left(\frac{\xi_3}{2^k}\right)\overline{\widehat{(f_4)_{k,n_4,l_4}}}(\xi_1+\xi_2+\xi_3)\,d\xi_1d\xi_2d\xi_3.
\end{split}
\]
By the support of functions, each term in the sum is non-zero only when 

\[
\begin{cases}
\xi_i\in[2^k(\frac{l_i}{2}-1),2^k(\frac{l_i}{2}+2)]  \text{ for } i=1,2,3; \\
|\xi_1-\xi_2|\lesssim 2^k, |\xi_3|\in[9\cdot 2^k,10\cdot 2^k);\\
\xi_1+\xi_2+\xi_3\in [2^k(\frac{l_4}{2}-1),2^k(\frac{l_4}{2}+2)].
\end{cases}
\]
These imply that 

\[
\begin{cases}
|l_2-l_1|\lesssim 1;\\
|l_3-9|\lesssim 1;\\
|l_4-(2l_1-18)|\lesssim 1.
\end{cases}
\]
In other words, among the four parameters $l_1,l_2,l_3,l_4$ only one is free, say $l_1$. Without loss of generality we can fix a dependence relation between $l_2,l_3,l_4$ and $l_1$. Then drop the cut-off functions by the Fourier expansion trick and ignore the fast decay terms so that $\La(f_1,f_2,f_3,f_4)$ becomes essentially as

\[
\sum_{\substack{k, l_1\\ n_1,n_2,n_3,n_4}}\int (f_1)_{k,n_1,l_1}(x)(f_2)_{k,n_2,l_2}(x)(f_3)_{k,n_3,l_3}(x)\overline{(f_4)_{k, n_4,l_4}}(x)\,dx.
\]
Since $(f_j)_{k,n_j,l_j}$ is almost supported in $I_{k,n_j}=[2^{-k}n_j,2^{-k}(n_j+1))$, there is not too much loss to assume $n_1= n_2=n_3=n_4$ due to the fast decay in other cases. Therefore, the original 4-linear form has been simplified to the following model form (we still use $\La$ to denote the model 4-linear form by an abuse of notation):

\begin{equation} \label{eq: model}
\La(f_1,f_2,f_3,f_4)=\sum_{(k,n,l)\in\ZZ^3}\int \prod_{j=1}^4 (f_j)_{k,n,l_j}(x)\, dx.
\end{equation} 
Here $l_1=l$, $l_2=l$, $l_3=18$ and $l_4=2l+18$.

We will prove directly that $T$ is of restricted weak type  (see \cite{MS} for the definition) when $(p_1, p_2, p_3)$ is in a smaller range $D_0:=\{(p_1,p_2,p_3): 1<p_1, p_2<2, \frac{1}{p_1}+\frac{1}{p_2}<\frac{3}{2}, p_3\in(1,\nf)\}$. More precisely, we will prove
\begin{theorem} \label{Thm: model}
  Let $(p_1,p_2,p_3)\in D_0$. For any measurable sets $F_1, F_2, F_3, F$ of finite measure, there exists measurable set $F'\subseteq F$ with $|F'|\ge \frac{1}{2}|F|$ such that $\La$ defined in \eqref{eq: model} satisfy 
  
\begin{equation} \label{eq: RWT}
  |\La(f_1,f_2,f_3,f_4)|\lesssim |F_1|^{\frac{1}{p_1}}|F_2|^{\frac{1}{p_2}}|F_3|^{\frac{1}{p_3}}|F'|^{\frac{1}{p'}}; 
\end{equation}
for every $|f_1|\le\hichi_{F_1}$, $|f_2|\le\hichi_{F_2}$, $|f_3|\le\hichi_{F_3}$ and $|f_4|\le\hichi_{F'}$. Here $\frac{1}{p'}:=1-(\frac{1}{p_1}+\frac{1}{p_2}+\frac{1}{p_3})$.
\end{theorem}

To prove Theorem \ref{Thm: model} we pick up an arbitrary finite subset $S\subset \ZZ^3$ and aim to obtain \eqref{eq: RWT} for

\begin{equation} \label{eq: model over S}
 \La_S(f_1,f_2,f_3,f_4):=\sum_{(k,n,l)\in S}\int \prod_{j=1}^4 (f_j)_{k,n,l_j}(x)\, dx,
\end{equation} 
provided the bound does not depend on the set $S$. We can also assume $|F|=1$ by dilation invariance. Next we make the geometric structure of $\La_S$ clearer. To each tuple $s=(k,n,l)\in \ZZ^3$ we assign a time-interval $I_s:=I_{k,n}$ and four frequency-intervals $\om_{s_j}$, $j\in\{1,2,3,4\}$, representing the localization of functions in the time-frequency space. More precisely, $I_s$ and $\om_{s_j}$'s satisfy:

\begin{align}
&(f_j)_{k,n,l_j}(x) \text{ is dominated by } \label{time interval} \\ 
&C_{N,M}\left(1+\frac{\text{dist}(x, I_s)}{|I_s|} \right)^{-N}\frac{1}{|I_s|}\int |f_j(y)| \left(1+\frac{|x-y|}{|I_s|} \right)^{-M} \,dy \notag \\
&\text{The Fourier transform of }(f_j)_{k,n,l_j} \text{is supported on } \om_{s_j}. \label{intersect} 
\end{align}
\begin{definition}
We call $s=(k,n,l)$ a \textit{4-tile} (or simply a tile) as it corresponds to 4 single-tiles $s_j:=I_s\times \om_{s_j}$, $j\in\{1,2,3,4\}$. Write $f_{s_j}:=f_{k,n,l_j}$ for simplicity.
\end{definition}

We can take finitely many sparse subsets of $S$ and transform $\om_{s_j}$'s by fixed affine mappings if needed (since only relative locations of Fourier supports matter) so that $I_s$ and $\om_{s_j}$'s enjoy nice geometric properties as follows:

\begin{align}
&\om_{s_1}=\om_{s_2}; \label{first property of tiles}\\
&|\om_{s_1}|=|\om_{s_3}|=|\om_{s_4}|=C|I_s|^{-1}; \\
&\text{dist}(\om_{s_1},\om_{s_4})=|\om_{s_1}|;\\
&c(\om_{s_1})>c(\om_{s_4}), \text{ where $c(I)$ is the center of the interval $I$};\\
&\{I_s\}_{s\in S} \text{ is a grid (defined below)};\\
&\{\om_{s_1}\cup\om_{s_4}\}_{s\in S} \text{ is a gird};\\
&\om_{s_i}\subsetneqq J \text{ for some } i\in\{1,4\}, J:=\om_{s'_1}\cup\om_{s'_2}\cup\om_{s'_4}, s'\in S \Rightarrow \label{all in}\\
&\om_{s_j}\subseteq J \text{ for all } j\in\{1,4\}. \notag 
\end{align}
Here a grid is defined as a set of intervals having the property that if two different elements intersect then one must contain the other and the larger interval is at least twice as long as the smaller one. See \cite{LT97} for a detailed construction of the time and frequency intervals. 

From now on we fix a finite set of tiles $S\subset\ZZ^3$ and assume the tiles satisfy \eqref{time interval}-\eqref{all in}. See Figure \ref{fig: compare} for a comparison between the tile structure of $T^{1,1}$ and that of BHT. 

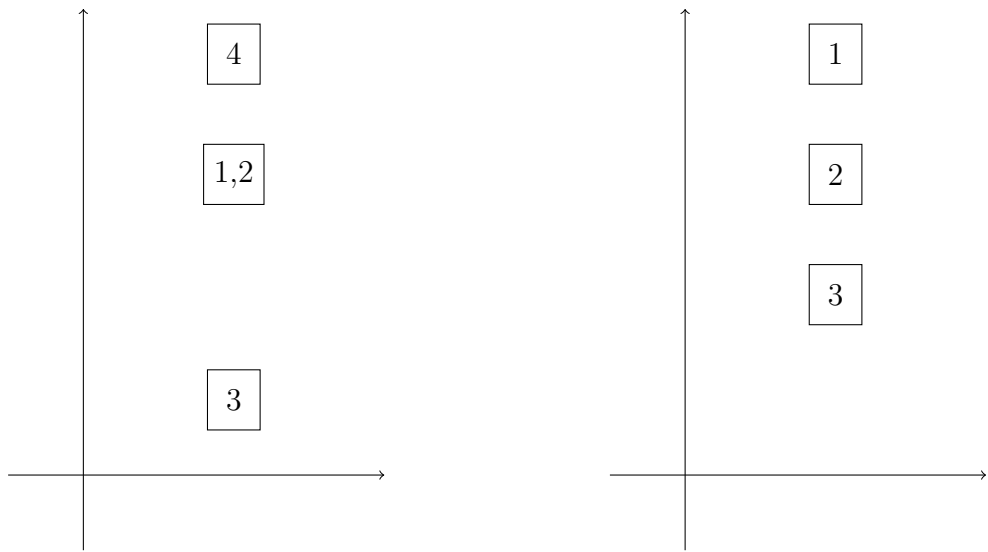
\begin{figure} [h]
\begin{center}
\begin{tikzpicture}

\draw[->] (7.0,0) -- (12,0) coordinate (x axis);
\draw[->] (8,-1.0) -- (8,6.2) coordinate (y axis);

\node (rect) [rectangle, draw, minimum width=7mm, minimum height=8mm] at (10,2.4) {3};

\node (rect) [rectangle, draw, minimum width=7mm, minimum height=8mm] at (10,4) {2};

\node (rect) [rectangle, draw, minimum width=7mm, minimum height=8mm] at (10,5.6) {1};

\draw[->] (-1.0,0) -- (4,0) coordinate (x axis);
\draw[->] (0,-1.0) -- (0,6.2) coordinate (y axis);

\node (rect) [rectangle, draw, minimum width=7mm, minimum height=8mm] at (2,1) {3};

\node (rect) [rectangle, draw, minimum width=7mm, minimum height=8mm] at (2,4) {1,2};

\node (rect) [rectangle, draw, minimum width=7mm, minimum height=8mm] at (2,5.6) {4};

\end{tikzpicture}
\caption{$4$-tile of $T^{1,1}$ vs. tri-tile of BHT}
\label{fig: compare}
\end{center}
\end{figure}

Theorem \ref{Thm: model} has been reduced to the following theorem. 
\begin{theorem} \label{Thm: model2}
 Let $p>1$ be arbitrary. Given any $(p_1,p_2,p_3)\in D_0$ with $p_3\ge p$ and any sets of finite measure $F_1,F_2,F_3,F$ with $|F|=1$, there exists $F'\subseteq F$ with $|F'|\ge\frac{1}{2}$ such that
 
\[
  |\La_S(f_1,f_2,f_3,f_4)|\lesssim |F_1|^{\frac{1}{p_1}}|F_2|^{\frac{1}{p_2}}|F_3|^{\frac{1}{p_3}} 
\]
for every $|f_1|\le\hichi_{F_1}$, $|f_2|\le\hichi_{F_2}$, $|f_3|\le\hichi_{F_3}$ and $|f_4|\le\hichi_{F'}$.
\end{theorem}

\section{Proof of Theorem \ref{thm: homo}} \label{section: pf}
\setcounter{equation}0

In this section we prove Theorem \ref{Thm: model2} and hence Theorem \ref{thm: homo}, using some propositions whose proof will be given in subsequent sections. Fix $p>1$, $(p_1,p_2,p_3)\in D_0=\{(p_1,p_2,p_3): 1<p_1, p_2<2, \frac{1}{p_1}+\frac{1}{p_2}<\frac{3}{2}, p_3\in(1,\nf)\}$ with $p_3>p$, and measurable sets $F_1, F_2, F_3, F$ with $|F|=1$. Let $\tM$ denote the maximal operator. Define the exceptional set $$\Om:=\left(\bigcup_{j=1}^2\left\{x:\tM(\hichi_{F_j})(x)>C|F_j|\right\}\right) \bigcup\left\{x:\tM(\hichi_{F_3})(x)>C|F_3|^{\frac{1}{p}}\right\}.$$
Then $|\Om|\le\frac{1}{4}$ when $C$ is large enough. Set $F':=F\setminus \Om$ so that $|F'|\ge\frac{1}{2}$. For any dyadic number $\mu\ge 1$, define 

\begin{equation} \label{def of S^d}
S^{\mu}:=\left\{s\in S: 1+\frac{\text{dist}(I_s,\Om^c)}{|I_s|}\simeq \mu \right\}.
\end{equation}
Then it suffices to obtain the estimate

\begin{equation} \label{eq: goal}
|\La_{S^{\mu}}(f_1,f_2,f_3,f_4)|\lesssim {\mu}^{-2}|F_1|^{\frac{1}{p_1}}|F_2|^{\frac{1}{p_2}}|F_3|^{\frac{1}{p_3}} \text{  for any dyadic } \mu\ge 1.
\end{equation}

The main idea to obtain \eqref{eq: goal} is to group the tiles in $S^{\mu}$ appropriately, aiming to establish orthogonality among groups. The following definitions are needed. 

\begin{definition}
Let $j\in\{1,4\}$. Given two 4-tiles $s$ and $s'$, we write $s_j<s'_j$ if $I_s\subseteq I_{s'}$ and $\om_{s_j}\supseteq\om_{s'_j}$. We call $T\subseteq S$ a \textbf{$j$-tree} if there exists a $t\in T$ such that $s_j<t_j$ for all $s\in T$. $t$ is called the \textbf{top} of $T$ and denote $I_T:=I_t$. We call $T\subseteq S$ a \textbf{tree} (with top $t$) if for any $s\in T$ we have $I_s\subseteq I_t$ and $\om_{s_j}\supseteq\om_{t_j}$ for some $j\in\{1,4\}$.
\end{definition}
It is easy to see that any tree is a union of a $1$-tree and a $4$-tree.

\begin{definition} \label{def of s}
For any $P\subseteq S$ and $f\in \cs(\ZR)$, define 

\[
\begin{split}
&\textbf{size}_j (P,f):=\sup_{\substack{T\subseteq P\\T \text{ is a $4$-tree }}}\left(\frac{1}{|I_T|}\sum_{s\in T}\|f_{s_j}\|_2^2\right)^{\frac{1}{2}}, j=1 \text{ or }2;\\
&\textbf{size}_4 (P,f):=\sup_{\substack{T\subseteq P\\T \text{ is  a $1$-tree}}}\left(\frac{1}{|I_T|}\sum_{s\in T}\|f_{s_4}\|_2^2\right)^{\frac{1}{2}}.
\end{split}
\] 
\end{definition}

Sizes can be controlled using the proposition below, whose proof will be given in Section \ref{section: tree and size}.
\begin{proposition} \label{lemma: size}
Fix a dyadic number $\mu\ge 1$. For any $P\subseteq S^{\mu}$, $j\in\{1,2,4\}$ and $f\in \cs(\ZR)$,

\[
\text{size}_j(P,f)\lesssim_M \sup_{s\in P}\left(\frac{1}{|I_s|}\|f\|_{L^1(\mu I_s)}+{\mu}^{-M}\inf_{y\in \mu I_s}Mf(y) \right).
\]
\end{proposition}

If tiles form a tree, then we can control the corresponding 4-form by sizes, as suggested by the following proposition.
\begin{proposition} \label{lemma: single tree}
Let $T\subseteq S^{\mu}$ be a tree. Then

\[
|\La_T(f_1,f_2,f_3,f_4)|\lesssim \mu|I_T|\prod_{j\in\{1,2,4\}}\text{size}_j(T, f_j)|F_3|^{\frac{1}{p_3}}.
\]
\end{proposition}
\begin{proof}
First assume $T$ is a $1$-tree. By Cauchy-Schwartz inequality, we have

\begin{equation*}
\begin{split}
&|\La_T(f_1,f_2,f_3,f_4)|\le \int\sup_{s\in T}|(f_1)_{s_1}|\sup_{s\in T}|(f_2)_{s_2}|\left(\sum_{s\in T}|(f_3)_{s_3}|^2\right)^{\frac{1}{2}}\left(\sum_{s\in T}|(f_4)_{s_4}|^2\right)^{\frac{1}{2}}\\
&\le|I_T|\sup_{s\in T}\|(f_1)_{s_1}\|_{\infty}\sup_{s\in T}\|(f_2)_{s_2}\|_{\infty}\left(\frac{1}{|I_T|}\sum_{s\in T}\|(f_3)_{s_3}\|_2^2\right)^{\frac{1}{2}}\left(\frac{1}{|I_T|}\sum_{s\in T}\|(f_4)_{s_4}\|_2^2\right)^{\frac{1}{2}}.
\end{split}
\end{equation*}

Using the structure of the $1$-tree and the definition of $S^{\mu}$, 

\begin{equation}
\left(\frac{1}{|I_T|}\sum_{s\in T}\|(f_3)_{s_3}\|_2^2\right)^{\frac{1}{2}}\lesssim \mu\min\{1, |F_3|^{\frac{1}{p}} \}\le \mu|F_3|^{\frac{1}{p_3}}
\end{equation}

Combine the above two estimates and can bound $|\La_T(f_1,f_2,f_3,f_4)|$ by

\[
\mu|I_T|\sup_{s\in T}\|(f_1)_{s_1}\|_{\infty} \sup_{s\in T}\|(f_2)_{s_2}\|_{\infty} \text{ size}_4(T,f_4) |F_3|^{\frac{1}{p_3}}.
\]

It remains to prove that for $i=1$ or $i=2$, $\|(f_i)_{s_i}\|_{\infty}\lesssim \text{size}_i(T,f_i)$ for any $s\in T$. We will only consider $i=1$ case as the other case can be handled similarly. We just need to prove the estimate 

\begin{equation} \label{1-size}
\|(f_1)_{s_1}\|_{\infty}\lesssim\|(f_1)_{s_1}\|_2|I_s|^{-\frac{1}{2}}
\end{equation}
since $\{s\}$ is a 4-tree.
To prove \eqref{1-size}, recall for $s=(k,n,l)$, $(f_1)_{s_1}(x)=\hichi_{I_{k,n}}^*(x)f_1*\psi_{k,l}(x)$, where $\psi_{k,l}(x)=2^k\psi (2^kx)e^{-2\pi i \frac{l}{2}x}$. Let $b$ be a real number such that $|\frac{l}{2}-b|=2^k$ and define $\widetilde{(f_1)_{s_1}}(x):=e^{2\pi ibx}(f_1)_{s_1}(x)$. Then $\widetilde{(f_1)_{s_1}}'(x)=\gamma (f_1)_{s_1}(x)$ for some $\gamma\lesssim 2^k$. Hence 

$$
\|(f_1)_{s_1}\|_{\infty}=\|\widetilde{(f_1)_{s_1}}\|_{\infty}\lesssim\sqrt{\|\widetilde{(f_1)_{s_1}}\|_2\|\widetilde{(f_1)_{s_1}}'\|_2}\lesssim 2^{\frac{k}{2}}\|(f_1)_{s_1}\|_2\lesssim \|(f_1)_{s_1}\|_2|I_s|^{-\frac{1}{2}}
$$ as desired.

Now assume $T$ is a $4$-tree. By similar arguments, we have

\[
\begin{split}
&|\La_T(f_1,f_2,f_3,f_4)|\le \int\left(\sum_{s\in T}|(f_1)_{s_1}|^2\right)^{\frac{1}{2}}\sup_{s\in T}|(f_2)_{s_2}|\left(\sum_{s\in T}|(f_3)_{s_3}|^2\right)^{\frac{1}{2}}\sup_{s\in T}|(f_4)_{s_4}|\\
&\le|I_T|\left(\frac{1}{|I_T|}\sum_{s\in T}\|(f_1)_{s_1}\|_2^2\right)^{\frac{1}{2}}\sup_{s\in T}\|(f_2)_{s_2}\|_{\infty}\left(\frac{1}{|I_T|}\sum_{s\in T}\|(f_3)_{s_3}\|_2^2\right)^{\frac{1}{2}}\sup_{s\in T}\|(f_4)_{s_4}\|_{\infty}\\
&\lesssim \mu|I_T|\text{ size}_1(T,f_1) \sup_{s\in T}\|(f_2)_{s_2}\|_{\infty} \sup_{s\in T}\|(f_4)_{s_4}\|_{\infty} |F_3|^{\frac{1}{p_3}}\\
&\lesssim \mu|I_T|\prod_{j\in\{1,2,4\}}\text{size }_j(T, f_j)|F_3|^{\frac{1}{p_3}}.
\end{split}
\]
This finishes the proof of Proposition \ref{lemma: single tree}.
\end{proof}

The following proposition provides the algorithm to select trees and group tiles. 

\begin{proposition} \label{lemma: org}
Let $f\in L^2$. Suppose for some $j\in\{1,2,4\}$ and $P\subseteq S$, we have

\[
\text{size}_j(P,f)\le \sigma\|f\|_2 \text{ for some dyadic number } \sigma=2^n, n\in \ZZ.
\]
Then we can decompose $P=P'\cup P''$ such that 

\begin{equation} \label{eq: half}
\text{size}_j(P',f)\le\ \frac{\sigma}{2}\|f\|_2 
\end{equation}
and $P''$ is a union of trees $T$ in some collection $\ff$ with $\sum_{T\in\ff}|I_T|\lesssim \frac{1}{\sigma^2}$.
\end{proposition}
The proof of this organization proposition will be postponed to Section \ref{section: org}. 

Now we are ready to prove our goal \eqref{eq: goal}. By Proposition \ref{lemma: size} and the definition of $S^{\mu}$, we have

\begin{equation}
\text{size}_j (S^{\mu}, f_j)\lesssim
\begin{cases}
\mu|F_j| \text{ when } j=1,2; \\
{\mu}^{-M} \text{ for any large } M>0 \text{ when } j=4.\\
\end{cases}
\end{equation}

Iterate the organization algorithm Proposition \ref{lemma: org} for all $j=1,2,4$ simultaneously, and we can decompose $S^{\mu}$ as

$$
S^{\mu}=\bigcup_{\substack{\sigma \text{ is a }\\\text{dyadic number}}} S_{\sigma},
$$
where

\begin{equation}
\text{size}_j (S_{\sigma}, f_j)\lesssim
\begin{cases}
\min\{\mu|F_j|,\sigma |F_j|^{\frac{1}{2}} \}\text{ when } j=1,2; \\
\min\{\mu^{-M},\sigma \}\text{ for any large } M>0 \text{ when } j=4,\\
\end{cases}
\end{equation}
and $S_{\sigma}=\cup_{T\in \ff_{\sigma}} T$ is a union of tees with $\sum_{T\in \ff_{\sigma}}|I_T|\lesssim \frac{1}{\sigma^2}$.

Using this decomposition and the estimate on a single tree (Proposition \ref{lemma: single tree}), we see that 

\[
\begin{split}
|\La_{S^{\mu}}(f_1,&f_2,f_3,f_4)|\lesssim\sum_{\sigma \text{ is dyadic}}\sum_{T\in\ff_{\sigma}} |\La_T(f_1,f_2,f_3,f_4)|\\
&\lesssim \mu\sum_{\sigma}\sum_{T\in\ff_{\sigma}}|I_T|\prod_{j\in\{1,2,4\}}\text{size }_j(T, f_j)|F_3|^{\frac{1}{p_3}}\\
&\lesssim {\mu}^{3} |F_3|^{\frac{1}{p_3}}\sum_{\sigma}\frac{1}{\sigma^2} \min\{|F_1|,\sigma |F_1|^{\frac{1}{2}} \} \min\{|F_2|,\sigma |F_2|^{\frac{1}{2}} \} \min\{{\mu}^{-M},\sigma \}. 
\end{split}
\]

Apply the elementary inequality $\min\{X,Y\}\le X^{\theta}Y^{1-\theta}$, and we can bound $|\La_{S^{\mu}}(f_1,f_2,f_3,f_4)|$ by

\[
\begin{split}
&{\mu}^{3} |F_3|^{\frac{1}{p_3}}\sum_{\sigma}\frac{1}{\sigma^2} \sigma^{2\left(1-\frac{1}{p_1}\right)+2\left(1-\frac{1}{p_2}\right)}|F_1|^{\frac{1}{p_1}}|F_2|^{\frac{1}{p_2}} \min\{{\mu}^{-M},\sigma \}\lesssim {\mu}^{-2}|F_1|^{\frac{1}{p_1}}|F_2|^{\frac{1}{p_2}}|F_3|^{\frac{1}{p_3}},
\end{split}
\]
where we used the fact $\frac{1}{p_1}+\frac{1}{p_2}<\frac{3}{2}$ in the last inequality. This proves \eqref{eq: goal}.
 
\section{Size Estimates} \label{section: tree and size}
\setcounter{equation}0

In this section, we prove Proposition \ref{lemma: size}. The proofs of some variants of this proposition already appear in \cite{DL} and \cite{MS}. For the convenience of the reader, we include the details here. First we need the following lemma which is another form of the John-Nirenberg inequality.
\begin{lemma} \label{J-N}
For any $P\subseteq S$ and $f\in \cs(\ZR)$,

\[
\begin{split}
&\text{size}_j(P,f)\lesssim\sup_{\substack{T\subseteq P\\T \text{is a $4$-tree}}}\frac{1}{|I_T|}\left\|\left(\sum_{s\in T}\frac{\|f_{s_j}\|_2^2}{|I_s|}\hichi_{I_s}\right)^{\frac{1}{2}}\right\|_{1,\infty},~ j\in\{1,2\},\\
&\text{size}_4(P,f)\lesssim\sup_{\substack{T\subseteq P\\T \text{is a $1$-tree}}}\frac{1}{|I_T|}\left\|\left(\sum_{s\in T}\frac{\|f_{s_j}\|_2^2}{|I_s|}\hichi_{I_s}\right)^{\frac{1}{2}}\right\|_{1,\infty}.
\end{split}
\]
\end{lemma}
\begin{proof}
Fix $j\in\{1,2,4\}$, $P\subseteq S$ and $f\in \cs(\ZR)$. Let $T\subseteq P$ be an $i$-tree for some $i\in\{1,4\}$ with $i\neq j$ such that 

\[
\text{size}_j (P,f)=\left(\frac{1}{|I_T|}\sum_{s\in T}\|f_{s_j}\|_2^2\right)^{\frac{1}{2}}
\]
For simplicity write $a_s:=\|f_{s_j}\|_2$ for $s\in T$ and we aim to show 

\begin{equation} \label{goal in J-N}
\left(\frac{1}{|I_T|}\sum_{s\in T}{a_s}^2\right)^{\frac{1}{2}}\lesssim \frac{1}{|I_T|}\left\|\left(\sum_{s\in T}\frac{{a_s}^2}{|I_s|}\hichi_{I_s}\right)^{\frac{1}{2}}\right\|_{1,\infty}.
\end{equation}
Denote the left-hand side (LHS) and the right-hand side (RHS) of \eqref{goal in J-N} by $A$ and $B$, respectively. Let $C$ be a large constant and define the set 

\begin{equation} \label{def of E}
E:=\left\{x:\left(\sum_{s\in T}\frac{{a_s}^2}{|I_s|}\hichi_{I_s}(x)\right)^{\frac{1}{2}}>CB\right\}\subseteq I_T.
\end{equation}
 By the definition of weak 1 norm,
 
\begin{equation} \label{measure of E}
|E|\leq\frac{B|I_T|}{CB}=\frac{|I_T|}{C}
\end{equation}
Write $E$ as a joint union of intervals $E=\bigcup_{I^m\in\cj^M} I^m$, where $\cj^M$ is the set of maximal elements in

\begin{equation} \label{def of J}
\cj:=\left\{I=I_{s_0} \text{ for some } s_0\in T: \left(\sum_{s\in T, I_s\supseteq I}{a_s}^2|I_s|^{-1}\right)^{\frac{1}{2}}>CB\right\}.
\end{equation}
By the definition of $A$,

\begin{equation} \label{A^2I_T}
A^2|I_T|=\sum_{s\in T}{a_s}^2=\int_E \sum_{s\in T}\frac{{a_s}^2}{|I_s|}\hichi_{I_s} +\int_{I_T\backslash E} \sum_{s\in T}\frac{{a_s}^2}{|I_s|}\hichi_{I_s}=:H+K.
\end{equation}
Use the decomposition $E=\bigcup_{I^m\in\cj^M} I^m$ to split $H$ further as

\begin{equation} \label{split H}
H=\sum_{I^m\in\cj^M}\int_{I^m} \sum_{\substack{s\in T, I_s\supsetneqq I^m}}\frac{{a_s}^2}{|I_s|}\hichi_{I_s} +\sum_{I^m\in\cj^M}\int_{I^m} \sum_{\substack{s\in T, I_s\subseteq I^m}}\frac{{a_s}^2}{|I_s|}\hichi_{I_s}=:H_1+H_2.
\end{equation}
Since each $I^m$ is maximal in $\cj$ defined by \eqref{def of J},

\begin{equation} \label{H_1}
H_1\leq\sum_{I^m\in\cj^M}(CB)^2|I^m|=(CB)^2|E|\leq(CB)^2|I_T|.
\end{equation}
For each $I^m\in \cj^M$, $\{s\in T: I_s\subseteq I^m\}$ is still an $i$-tree by the grid structure. So the definition of $\textit{size}_j(P,f)$ and \eqref{measure of E} give

\begin{equation} \label{H_2}
H_2=\sum_{I^m\in\cj^M}|I^m|\left(\frac{1}{|I^m|}\sum_{s\in T, I_s\subseteq I^m}{a_s}^2\right)\leq\sum_{I^m\in\cj^M}|I^m|A^2=A^2|E|\leq A^2\frac{|I_T|}{C}
\end{equation}

\noindent Since the integrand in $K$ is dominated by $CB$ by \eqref{def of E}, we have 

\begin{equation} \label{estimate on K}
K\leq (CB)^2|I_T|.
\end{equation}
Putting \eqref{A^2I_T}-\eqref{estimate on K} together, we obtain

\begin{equation}
A^2|I_T|=H_1+H_2+K\leq (CB)^2|I_T|+A^2\frac{|I_T|}{C}+(CB)^2|I_T|,
\end{equation}
from which we obtain $A\lesssim B$. This proves \eqref{goal in J-N} and thus Lemma \ref{J-N}.

\end{proof}
We now turn to the proof of Proposition \ref{lemma: size}. Without loss of generality, assume $j=1$. By Lemma \ref{J-N}, it suffices to show for any $4$-tree $T$,

\begin{equation} \label{eq: BMO}
\left\|\left(\sum_{s\in T}\frac{\|f_{s_1}\|_2^2}{|I_s|}\hichi_{I_s}\right)^{\frac{1}{2}}\right\|_{1,\infty} \lesssim_M \|f\|_{L^1(\mu I_T)}+{\mu}^{-M}\inf_{y\in \mu I_T}\tM f(y) |I_T|.
\end{equation}

\noindent Write $f=f\hichi_{\mu I_T}+f\hichi_{(\mu I_T)^c}$. LHS of \eqref{eq: BMO} is bounded by

\[
\left\|\left(\sum_{s\in T}\frac{\|(f\hichi_{ \mu I_T})_{s_1}\|_2^2}{|I_s|}\hichi_{I_s}\right)^{\frac{1}{2}}\right\|_{1,\infty}+\left\|\left(\sum_{s\in T}\frac{\|(f\hichi_{(\mu I_T)^c})_{s_1}\|_2^2}{|I_s|}\hichi_{I_s}\right)^{\frac{1}{2}}\right\|_1=:I+II.
\]

\noindent By the conditions \eqref{first property of tiles}-\eqref{all in} of the tiles, in a $4$-tree, $s_1$ tiles are Littlewood-Paley pieces as illustrated in Figure \ref{fig: local L-P}. Thus term $I$ is bounded by $C\|f\|_{L^1(\mu I_T)}$ since the discrete square-function operator is of weak type $(1,1)$ by the $L^2$ estimate and Calder\'on-Zygmund decomposition. 

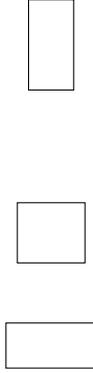
\begin{figure} [h]
\begin{center}
\begin{tikzpicture}

%\draw[->] (-1.0,0) -- (4,0) coordinate (x axis);
%\draw[->] (0,-1.0) -- (0,6) coordinate (y axis);

\node (rect) [rectangle, draw, minimum width=6mm, minimum height=12mm] at (2,4) {};

\node (rect) [rectangle, draw, minimum width=9mm, minimum height=8mm] at (2,1.5) {};

\node (rect) [rectangle, draw, minimum width=12mm, minimum height=6mm] at (2,0) {};

\end{tikzpicture}
\caption{$s_1$ tiles in a $4$-tree}
 \label{fig: local L-P}
\end{center}
\end{figure}

Using the fact $l^2$ norm is no more than $l^1$ norm, we estimate $II$ by 

$$
\sum_{s\in T}\|(f\hichi_{(\mu I_T)^c})_{s_1}\|_2|I_s|^{\frac{1}{2}}.
$$
It remains to show 

\begin{equation} \label{eq: II}
\sum_{s\in T}\|(f\hichi_{(\mu I_T)^c})_{s_1}\|_2|I_s|^{\frac{1}{2}}\lesssim_M {\mu}^{-M}\inf_{y\in \mu I_T}\tM f(y) |I_T|.
\end{equation}
Using \eqref{time interval} we see that control the function $|(f\hichi_{(\mu I_T)^c})_{s_1}(x)|$ is bounded above by

\[
\left(1+\frac{\text{dist}(I_s,(\mu I_T)^c)}{|I_s|}\right)^{-N}\left(1+\frac{\text{dist}(x,I_s)}{|I_s|}\right)^{-N}\inf_{y\in \mu I_T}\tM f(y).
\]
Hence $\sum_{s\in T}\|(f\hichi_{(\mu I_T)^c})_{s_1}\|_2|I_s|^{\frac{1}{2}}$ is dominated by

\[
\inf_{y\in \mu I_T}\tM f(y)\sum_{s\in T}|I_s|\left(1+\frac{\text{dist}(I_s,( \mu I_T)^c)}{|I_s|}\right)^{-N}\lesssim_M {\mu}^{-M}\inf_{y\in \mu I_T} \tM f(y) |I_T|,
\] as desired.
 
\section{Organizing Tiles} \label{section: org}
\setcounter{equation}0
We provide the proof of Proposition \ref{lemma: org} in this section. Without loss of generality, let $j=1$. By the assumptions of Proposition \ref{lemma: org},

\begin{equation}
\sup_{\substack{T\subseteq P\\T \text{ is a $4$-tree }}}\left(\frac{1}{|I_T|}\sum_{s\in T}\|f_{s_1}\|_2^2\right)^{\frac{1}{2}}\le\sigma\|f\|_2.
\end{equation}

Now we begin the tree selection algorithm. Initially set $S_0=P$ and $\ff=\emptyset$. Let

\begin{equation} \label{def of ff0}
\ff_0=\left\{T\subseteq S_0: T \text{ is a } 4\text{-tree such that } \left(\frac{1}{|I_T|}\sum_{s\in T}\|f_{s_1}\|_2^2\right)^{\frac{1}{2}} \ge \frac{\sigma}{2}\|f\|_2 \right\}.
\end{equation}
If $\ff_0\neq \emptyset$, then take $T_1$ to be the $4$-tree in $\ff_0$ with top $t$ such that $c(\om_{t_4})\ge c(\om_{t'_4})$ for any other $T\in \ff_0$ with top $t'$. Let 

\[
\begin{cases}
T_1^{(4)}:=\text{maximal } 4\text{-tree in }S_0 \text{ with top }t,\\ 
T_1^{(1)}:=\text{maximal } 1\text{-tree in }S_0 \text{ with top }t,\\ 
T_1^*:=T_1^{(1)}\cup T_1^{(4)} ~(\text{This is a tree with top } t).
\end{cases}
\]
Update $S_0$ and $\ff$ by setting $S_0:=S_0\setminus T_1^*$ and $\ff:=\ff\cup \{T_1^*\}$.

Repeat this algorithm until there is no $4$-tree in the updated $S_0$ satisfying

$$
\left(\frac{1}{|I_T|}\sum_{s\in T}\|f_{s_1}\|_2^2\right)^{\frac{1}{2}} \ge \frac{\sigma}{2}\|f\|_2.
$$
When the algorithm terminates, we obtain
\[
\begin{split}
&S_0=P\setminus \{T_1^*, T_2^*,\dots, T_l^* \},\\
&\ff=\{T_1^*, T_2^*,\dots, T_l^*\}.
\end{split}
\]

Simply let $P'=S_0$ and $P''=\cup_{T\in\ff}T$. Then Clearly $\text{size}_1(P',f)\le\frac{\sigma}{2}\|f\|_2$. 

Now we turn to the proof of $\sum_{T\in\ff}|I_T|\lesssim \frac{1}{\sigma^2}$. We can assume that each $T\in\ff$ is a $4$-tree. By the definition of $\ff_0$ \eqref{def of ff0}, for any $T\in\ff$,

\begin{equation} \label{property of T in ff0}
\left(\frac{1}{|I_T|}\sum_{s\in T}\|f_{s_1}\|_2^2\right)^{\frac{1}{2}} \ge \frac{\sigma}{2}\|f\|_2.
\end{equation}
Therefore,

$$
\sum_{T\in\ff}|I_T|\lesssim \frac{1}{\sigma^2\|f\|_2^2}\sum_{T\in\ff}\sum_{s\in T} \|f_{s_1}\|_2^2.
$$

It will suffice to prove 

\begin{equation} \label{mid goal in energy estimate}
\sum_{T\in \ff}\sum_{s\in T}\|f_{s_1}\|_2^2\lesssim \|f\|_2^2.
\end{equation}
For each 4-tile $s$, define an operator $A_s$ by $A_sf(x)=f_{s_1}(x)$. By Cauchy-Schwartz inequality, 

$$
\sum_{T\in \ff}\sum_{s\in T}\|f_{s_j}\|_2^2=\left\langle\sum_{T\in\ff}\sum_{s\in T}A_s^*A_sf,f\right\rangle\le\left\|\sum_{T\in\ff}\sum_{s\in T}A_s^*A_sf \right\|_2\|f\|_2.
$$
Hence \eqref{mid goal in energy estimate} follows from the following estimate: 

\begin{equation} \label{eq: e2}
\left\|\sum_{T\in\ff}\sum_{s\in T}A_s^*A_sf \right\|_2\lesssim \left(\sum_{T\in \ff}\sum_{s\in T}\|f_{s_j}\|_2^2\right)^{\frac{1}{2}}.
\end{equation}
To prove \eqref{eq: e2}, write

\[
(\text{LHS of \eqref{eq: e2}})^2=\sum_{T,T'\in\ff}\sum_{\substack{s\in T \\ s'\in T'}}\left\langle A_s^*A_sf,A_{s'}^*A_{s'}f \right\rangle=I+II,
\]
where 

\[
\begin{cases}
I:=\sum_{T\neq T'\in\ff}\sum_{\substack{s\in T \\ s'\in T'}}\left\langle A_s^*A_sf,A_{s'}^*A_{s'}f \right\rangle,\\
II:=\sum_{T\in\ff}\sum_{\substack{s,s'\in T}}\left\langle A_s^*A_sf,A_{s'}^*A_{s'}f \right\rangle.
\end{cases}
\]
Therefore, \eqref{eq: e2} follows from the estimate 

\begin{equation} \label{goal for bfI}
\max\{I,II\}\lesssim \sum_{T\in \ff}\sum_{s\in T}\|f_{s_1}\|_2^2.
\end{equation}
We will only provide the estimate for $I$, as $II$ is easier to control so we omit the proof. Apply Cauchy-Schwartz inequality, 

$$
I \le\sum_{T\neq T'\in\ff}\sum_{\substack{s\in T \\ s'\in T'}}\|A_sf\|_2\|A_sA_{s'}^*\|\|A_{s'}f\|_2.
$$
Hence \eqref{goal for bfI} is a consequence of the inequality below.

\begin{equation} \label{eq: goal before AA}
\sum_{T\neq T'\in\ff}\sum_{\substack{s\in T \\ s'\in T'}}\|A_sf\|_2\|A_sA_{s'}^*\|\|A_{s'}f\|_2 \lesssim \sum_{T\in \ff}\sum_{s\in T}\|f_{s_1}\|_2^2.
\end{equation}

The following estimate for $\|A_sA_{s'}^*\|$ is the key to sum up all the terms in the LHS of \eqref{eq: goal before AA}.
\begin{claim}
$\|A_sA_{s'}^*\|\neq 0$ only when $\om_{s_1}\cap\om_{s'_1}\neq \emptyset$. Moreover,

\begin{equation} \label{eq: schur}
\|A_sA_{s'}^*\|\lesssim_N\frac{|I_{s'}|^{\frac{1}{2}}}{|I_s|^{\frac{1}{2}}}\left(1+\frac{\text{dist}(I_s,I_{s'})}{|I_s|}\right)^{-N} \text{\quad if \quad $\om_{s_1}\subseteq\om_{s'_1}$.}
\end{equation} 
\end{claim} 
\begin{proof}
Write $A_sA_{s'}^*f(x)=\int K(x,y)f(y)\,dy$, where $K(x,y)=\hichi_{I_s}^*(x)\hichi_{I_{s'}}^*(y)\widetilde{\psi_{s_j'}}*\psi_{s_j}(x-y)$, $\psi_{s_j}:=\psi_{k,l_j}$ for $s=(k,n,l)$ and $\widetilde{g}(x):=\overline{g(-x)}$ for any function $g$. Note that $\widetilde{\psi_{s_j'}}*\psi_{s_j}(t)=\int \overline{\widehat{\psi_{s'}}}(\xi)\widehat{\psi_s}(\xi)e^{2\pi i\xi t}\, d\xi$ is non-zero only when $\om_{s_j}\cap\om_{s'_j}\neq \emptyset$ by \eqref{f} and \eqref{intersect}. Assume $\om_{s_1}\subseteq\om_{s'_1}$. By the definitions of $\hichi_I^*$ \eqref{chi I} and $\psi_{k,l}$ and using the triangle inequality $(1+|a|)^{-1}+(1+|b|)^{-1}\leq (1+|a+b|)^{-1}$, 

\[
\begin{split}
|K(x,y)|&\lesssim_N\left(1+\frac{\text{dist}(x,I_s)}{|I_s|}\right)^{-2N}\left(1+\frac{\text{dist}(y,I_{s'})}{|I_{s'}|}\right)^{-N}\\
&\qquad\frac{1}{|I_s||I_{s'}|}\int \left(1+\frac{|x-y-z|}{|I_{s'}|}\right)^{-2N}\left(1+\frac{|z|}{|I_s|}\right)^{-N}\, dz\\
&\lesssim_N\left(1+\frac{\text{dist}(I_s,I_{s'})}{|I_s|}\right)^{-N}\frac{1}{|I_s|}\left(1+\frac{\text{dist}(x,I_{s})}{|I_{s}|}\right)^{-N}.
\end{split}
\]
Hence 

\begin{equation} \label{eq: schur1}
\int|K(x,y)|dx\lesssim_N\left(1+\frac{\text{dist}(I_s,I_{s'})}{|I_s|}\right)^{-N}.
\end{equation}
Similarly,

\begin{equation} \label{eq: schur2}
\int|K(x,y)|dy\lesssim_N\left(1+\frac{\text{dist}(I_s,I_{s'})}{|I_s|}\right)^{-N}\frac{|I_{s'}|}{|I_s|}.
\end{equation}
\eqref{eq: schur1} and \eqref{eq: schur2} imply \eqref{eq: schur} by Schur's lemma.
\end{proof}
By the claim and symmetry, in the proof of \eqref{eq: goal before AA} we will assume without loss of generality $\om_{s_1}\subseteq\om_{s'_1}$. We will also assume that  $\om_{s_1}\subsetneqq\om_{s'_1}$, as the case $\om_{s_1}=\om_{s'_1}$ can be handled the same way. Under these assumptions, \eqref{eq: goal before AA} has been reduced to

\begin{equation} \label{reduced goal for bI}
\sum_{T\neq T'\in\ff}\sum_{\substack{s\in T, s'\in T'\\\om_{s_1}\subsetneqq\om_{s'_1}}}\|A_sf\|_2\|A_sA_{s'}^*\|\|A_{s'}f\|_2 \lesssim \sum_{T\in \ff}\sum_{s\in T}\|f_{s_1}\|_2^2.
\end{equation}

Since $\{s\}$ is a $4$-tree and size$_1(f,P)\le\sigma\|f\|_2$,

\begin{equation} \label{Af 1}
\|A_sf\|_2\le |I_s|^{\frac{1}{2}}\sigma\|f\|_2.
\end{equation}
Also notice that by \eqref{property of T in ff0}

\begin{equation} \label{Af 2}
\sigma\|f\|_2\lesssim \left(|I_T|^{-1}\sum_{s_0\in T}\|f_{(s_0)_1}\|_2^2\right)^{\frac{1}{2}}.
\end{equation}
 
Combine \eqref{Af 1} and \eqref{Af 2}, and we see that 

\begin{equation} \label{A_sf 1}
\|A_sf\|_2\lesssim |I_s|^{\frac{1}{2}}|I_T|^{-\frac{1}{2}}\left(\sum_{s_0\in T}\|f_{(s_0)_1}\|_2^2\right)^{\frac{1}{2}}.
\end{equation}
Similarly,  

\begin{equation} \label{A_sf 2}
\|A_{s'}f\|_2\lesssim|I_{s'}|^{\frac{1}{2}}|I_T|^{-\frac{1}{2}}\left(\sum_{s_0\in T}\|f_{(s_0)_1}\|_2^2\right)^{\frac{1}{2}}.
\end{equation}
Using \eqref{A_sf 1} and \eqref{A_sf 2}, LHS of \eqref{reduced goal for bI} is bounded by

\[
\sum_{T\in\bT}\left(\sum_{s_0\in T}\|f_{(s_0)_1}\|_2^2\right)\left(\sum_{\substack{s\in T, T'\neq T\\s'\in T',\om_{s_1}\subsetneqq\om_{s'_1}}}|I_s|^{\frac{1}{2}}|I_{s'}|^{\frac{1}{2}}|I_T|^{-1}\|A_sA_{s'}^*\|\right).
\]
Therefore, \eqref{reduced goal for bI} will be established once we show that for any $T\in\ff$,

$$
\sum_{\substack{s\in T, T'\neq T\\s'\in T',\om_{s_1}\subsetneqq\om_{s'_1}}}|I_s|^{\frac{1}{2}}|I_{s'}|^{\frac{1}{2}}|I_T|^{-1}\|A_sA_{s'}^*\|\lesssim 1.
$$
By \eqref{eq: schur}, this can be reduced to the estimate that for any $T\in \ff$,

\begin{equation} \label{goal before observation}
\sum_{\substack{s\in T,T'\neq T\\s'\in T',\om_{s_1}\subsetneqq\om_{s'_1}}}\left(1+\frac{\text{dist}(I_s,I_{s'})}{|I_s|}\right)^{-N} |I_{s'}|\lesssim |I_T|.
\end{equation}

To prove \eqref{goal before observation}, we need a crucial observation.

\begin{figure} [h]
\begin{center}
\begin{tikzpicture}

\node (rect) [rectangle, draw, minimum width=8mm, minimum height=50mm] at (4,8) {$s'_1$};

\node (rect) [rectangle, draw, minimum width=8mm, minimum height=50mm] at (4,0) {$s'_4$};

\node (rect) [rectangle, draw, minimum width=15mm, minimum height=20mm] at (2,9.5) {$s_1$};

\node (rect) [rectangle, draw, minimum width=15mm, minimum height=20mm] at (2,6.5) {$s_4$};

\node (rect) [rectangle, draw, minimum width=70mm, minimum height=5mm] at (2,7) {\hspace{-5 cm}$t_1$};

\node (rect) [rectangle, draw, minimum width=70mm, minimum height=5mm] at (2,6) {\hspace{-5 cm} $t_4$};

\node (rect) [rectangle, draw, minimum width=55mm, minimum height=10mm] at (2,-1) {$t'_4$};

\end{tikzpicture}
\caption{a crucial geometric observation}
 \label{fig: crucial observation}
\end{center}
\end{figure}

\begin{claim} \label{crucial observation}
If $T_1\neq T_2 \in\ff$, $s\in T_1$, and $s'\in T_2$, then

$$
\om_{s_1}\subseteq \om_{s'_1} \Rightarrow I_{s'}\cap I_{T_1}=\emptyset.
$$
\end{claim}

\begin{proof}
Let $t$ and $t'$ denote the top of $T_1$ and $T_2$ respectively.
Assume otherwise $I_{s'}\cap I_{T_1}\neq\emptyset$. Then $I_{s'}\subseteq I_t$. By \eqref{all in} and the definition of tree, $\om_{s'_1}\supseteq \om_{s_4}\supseteq \om_{t_4}$. Then $T_1$ is selected before $T_2$ as $c(\om_{t_4})>c(\om_{t'_4})$. However, $s'_1<t_1$ indicates that $s'$ should be selected together with $T_1$ according to the algorithm (See Figure \ref{fig: crucial observation}). This contradicts with the assumption that $s'\in T_2$.

\end{proof}
Now we are ready to prove \eqref{goal before observation}. It is easy to see that

\[
\text{LHS of \eqref{goal before observation}}\lesssim \sum_{s\in T}\sum_{\substack{T'\neq T\\s'\in T',\om_{s_1}\subsetneqq\om_{s'_1}}}\int_{I_{s'}} \left(1+\frac{\text{dist}(I_s,x)}{|I_s|}\right)^{-N}\,dx.
\]
By Claim \ref{crucial observation}, $I_{s'}$'s are pairwise disjoint and the union of these intervals is contained in $(I_T)^c$. Therefore,

\[
\begin{split}
\sum_{s\in T}&\sum_{\substack{T'\neq T\\s'\in T',\om_{s_1}\subsetneqq\om_{s'_1}}}\int_{I_{s'}} \left(1+\frac{\text{dist}(I_s,x)}{|I_s|}\right)^{-N}\,dx\\
&\le\sum_{s\in T}\int_{(I_T)^c}\left(1+\frac{\text{dist}(I_s,x)}{|I_s|}\right)^{-N}\,dx\lesssim\sum_{s\in T}\left(1+\frac{\text{dist}(I_s,(I_T)^c)}{|I_s|}\right)^{-N}|I_s|
\end{split}
\]
Using the tree structure of $T$ and the grid structure of tiles, it is easy to see that

$$
\sum_{s\in T}\left(1+\frac{\text{dist}(I_s,(I_T)^c)}{|I_s|}\right)^{-N}|I_s|\lesssim |I_T|.
$$
This proves \eqref{goal before observation}.

\section{Telescoping} \label{section: telescoping}
\setcounter{equation}0
We prove Theorem \ref{thm:non-homo} by a telescoping argument. In what follows, $[x]$ will be used to denote the integer part of $x\in\ZR$. 

Since $k\ge N\ge 10\al/\beta$, $[\frac{\beta}{\al}k]$ is large and essentially we have

\[
T_N^{\al,\beta}(f_1,f_2,f_3)(x)=\sum_{k\ge N}H^{\al,k}(f_1,f_2)(x)f_{3}^{\beta,k}(x)=:A+B,
\]
where

\[
\begin{cases}
A:=\sum_{k\ge N}\left(\sum_{j=0}^{[(1-\frac{\beta}{\al})k]-1}\left(H^{\al,k-j}(f_1,f_2)(x)- H^{\al,k-j-1}(f_1,f_2)(x)\right) \right)f_{3}^{\beta,k}(x),\\
B:=\sum_{k\ge N}H^{\al,[\frac{\beta}{\al}k]}(f_1,f_2)(x)f_3^{\beta,k}(x) =\sum_{k\ge N}H^{\beta,k}(f_1,f_2)(x)f_{3}^{\beta,k}(x).
\end{cases}
\]

$B$ has a much better tile structure than $T^{\al,\beta}_N$: See Figure \ref{fig: before tele} and Figure \ref{fig: after tele} for a comparison. Since $B$ is a part of $T^{\beta,\beta}$ and the proof of Theorem \ref{thm: homo} is valid for any collection of scales $k$, boundedness of $B$ is obtained.

\begin{figure} [h]
\begin{center}
\begin{tikzpicture}

\draw[->] (-1.0,0) -- (4,0) coordinate (x axis);
\draw[->] (0,-1.0) -- (0,4) coordinate (y axis);

\node (rect) [rectangle, draw, minimum width=7mm, minimum height=8mm] at (2,1) {3};

\node (rect) [rectangle, draw, minimum width=7mm, minimum height=8mm] at (2,3) {1,2};

\end{tikzpicture}
\caption{tri-tile structure of $T^{\beta,\beta}$}
\label{fig: after tele}
\end{center}
\end{figure}
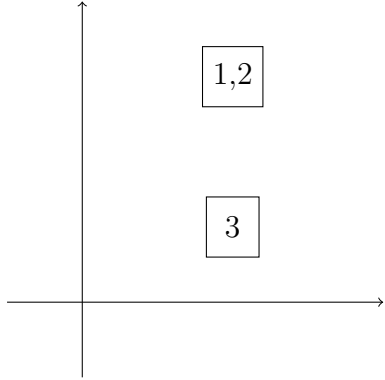

It remains to analyze the operator $A$. By a change of variable $k\to k+j$, we can write $A=I+II$, where

\[
\begin{cases}
I:=\sum_{k\ge N}\left(H^{\al,k}(f_1,f_2)(x)- H^{\al,k-1}(f_1,f_2)(x)\right) (\sum_{j=0}^{[(\frac{\al}{\beta}-1)k]}f_{3}^{\beta,k+j}(x)),\\
II:=\sum_{(k,j)\in P}\left(H^{\al,k}(f_1,f_2)(x)- H^{\al,k-1}(f_1,f_2)(x)\right) f_{3}^{\beta,k+j}(x).%,~ \text{$P$ is a finite set}%
\end{cases}
\]
Here $P$ is a finite set of indices, and $II$ should be considered as an error term, whose boundedness follows from H\"older and Lacey-Thiele's Theorem (\cite{LT97,LT99}). To prove the boundedness of the main term $I$, first note that 

\[
\sum_{j=0}^{[(\frac{\al}{\be}-1)k]}f_{3}^{\beta,k+j}(x)=f_{3}^{\al k}(x)-f_{3}^{\beta k}(x),
\]
where

\[
f^{l}(x):=\int \hat{f}(\xi)\phi_0\left(\frac{\xi}{2^l}\right)\,d\xi, ~l\in\ZR,
\]
for some bump function $\phi_0$ supported in $[-1,1]$.
Hence we can write $I$ as the difference of two parts:

\[
\begin{split}
I=&\sum_{k\ge 1}\left(H^{\al,k}(f_1,f_2)(x)- H^{\al,k-1}(f_1,f_2)(x)\right)f_{3}^{\al k}(x)-\\
&\sum_{k\ge 1}\left(H^{\al,k}(f_1,f_2)(x)- H^{\al,k-1}(f_1,f_2)(x)\right)f_{3}^{\beta, k}(x).
\end{split}
\]
Note that $H^{\al,k}(f_1,f_2)(x)- H^{\al,k-1}(f_1,f_2)(x)$ is a piece of BHT at scale $k$. Since $\al>\beta$ and $k>0$, the supports of $\widehat{{f_{3}^{\al k}}}$ and $\widehat{{f_{3}^{\beta k}}}$ are at most as large as $2^{\al k}$. We can introduce a fourth function and do the wave packet decomposition to $f_1,f_2,f_4$. Then the tiles associated with these functions have structures similar to that of the tri-tiles as in the study of BHT. Therefor, the proof of Theorem \ref{Thm:DL} given in \cite{DL} still applies to $I$, and we omit the details. This finishes the proof of Theorem \ref{thm:non-homo}.

\begin{acknowledgement}
The author would like to thank Prof. Xiaochun Li for helpful discussions on this topic. He also acknowledges the support from Gene H. Golub Fund of Mathematics Department at University of Illinois. 
\end{acknowledgement}

\end{document}